\documentclass[12pt,a4paper,leqno]{amsart}
\usepackage[utf8]{inputenc}
\usepackage{amssymb,amsmath,color,tikz}
\usepackage{comment}
\usepackage[all]{xy}

\newcommand{\CC}{{\mathbb C}}
\newcommand{\cE}{{\mathcal E}}

\newcommand{\cS}{{\mathcal S}}
\newcommand{\cZ}{{\mathcal Z}}
\newcommand{\cU}{{\mathcal U}}
\newcommand{\cF}{{\mathcal F}}
\newcommand{\cR}{{\mathcal R}}
\newcommand{\cQ}{{\mathcal Q}}
\newcommand{\cG}{{\mathcal G}}

\newcommand{\cO}{{\mathcal O}}
\newcommand{\cW}{{\mathcal W}}

\newcommand{\cSS}{\widetilde{\mathcal S}}

\newcommand {\PP}{\mathbb{P}}
\newcommand {\HH}{\mathbb{H}}
\newcommand{\cH}{{\mathcal H}}
\newcommand{\cL}{{\mathcal L}}

\newcommand{\ZZ}{\mathbb{Z}}
\newcommand{\KK}{\mathbb{K}}

\DeclareMathOperator{\Syz}{Syz}
\DeclareMathOperator{\GSyz}{{\cG}{\cS}yz}
\DeclareMathOperator{\SSyz}{{\cS}yz}
\DeclareMathOperator{\Tw}{Tw}
\DeclareMathOperator{\VB}{VB}

\DeclareMathOperator{\Spl}{Spl}
\DeclareMathOperator{\SSpl}{{\cS}pl}

\DeclareMathOperator{\Pic}{Pic}
\DeclareMathOperator{\End}{End}

\DeclareMathOperator{\rk}{rk}
\DeclareMathOperator{\id}{id}
\DeclareMathOperator{\Coh}{Coh}
\DeclareMathOperator{\Ext}{Ext}

\DeclareMathOperator{\Aut}{Aut}
\DeclareMathOperator{\Hilb}{Hilb}
\DeclareMathOperator{\Spec}{Spec}

\DeclareMathOperator{\eval}{eval}

\DeclareMathOperator{\Quot}{Quot}

\DeclareMathOperator{\Gm}{\mathbb{G}_{\text{m}}}

\newtheorem{theorem}{Theorem}[section]
\newtheorem{lemma}[theorem]{Lemma}

\newtheorem{proposition}[theorem]{Proposition}
\newtheorem{corollary}[theorem]{Corollary}
 \theoremstyle{definition}
\newtheorem{definition}[theorem]{Definition}
\newtheorem{example}[theorem]{Example}
\newtheorem{remark}[theorem]{Remark}
\newtheorem{notation}[theorem]{Notation}

\title[Moduli of syzygy bundles]{Moduli of generalized syzygy bundles}

\date{\today}

\begin{document}

 \author[B.\ Fantechi]{Barbara Fantechi} 
 \address{SISSA Via Bonomea 265, I-34136 Trieste, Italy}
  \email{fantechi@sissa.it, ORCID 0000-0002-7109-6818}.
  \author[R.\ M.\ Mir\'o-Roig]{Rosa M.\ Mir\'o-Roig} 
  \address{Facultat de
  Matem\`atiques i Inform\`atica, Universitat de Barcelona, Gran Via des les
  Corts Catalanes 585, 08007 Barcelona, Spain} \email{miro@ub.edu, ORCID 0000-0003-1375-6547}

\thanks{The first author has been partially supported by PRIN 2017 {\em Moduli theory and birational classification}} 
\thanks{The second author has been partially supported by the grant PID2019-104844GB-I00}

\begin{abstract} Given a vector bundle $F$ on a variety $X$ and $W\subset H^0(F)$ such that the evaluation map $W\otimes \cO_X\to F$ is surjective, its kernel $S_{F,W}$ is called generalized syzygy bundle.

Under mild assumptions, we construct a moduli space $\cG^0_U$ of simple generalized syzygy bundles, and show that the natural morphism $\alpha$ to the moduli of simple sheaves is a locally closed embedding. If moreover  $H^1(X,\cO_X)=0$, we find an explicit open subspace $\cG^0_V$ of $\cG^0_U$ where the restriction of $\alpha$ is an open embedding.

    In particular, if $\dim X\ge 3$ and $H^1(\cO_X)=0$, starting from an ample line bundle (or a simple rigid vector bundle) on $X$  we construct recursively open subspaces of moduli spaces of simple sheaves on $X$ that are smooth, rational, quasiprojective varieties.
\end{abstract}

\maketitle

\tableofcontents

\section{Introduction}

Throughout this paper $X$ will be a reduced connected  Gorenstein projective variety of pure dimension $n\ge 2$ defined over an algebraically closed field $\KK$.
Let $F$ be a vector bundle on $X$ generated by its global sections; its {\em syzygy bundle} is the kernel $S$ of the (surjective) evaluation map $H^0(X,F)\otimes{\mathcal O}_X\to F$.

 Arising in a variety of geometric and algebraic problems, syzygy bundles $S$  have been intensively studied from different points of view ranging from the syzygies of $X$ (and the $N_p$ properties in the sense of Green) to questions of tight closure. In particular, when $F=\cO_X(L)$ is globally generated,  many efforts have been invested in knowing whether  the associated syzygy bundle is stable with respect to some polarization and recent results (see, for instance, \cite{BP}, \cite{CL}, \cite{HMP}, \cite{MR} and \cite{T})
   have led to the so-called Ein-Lazarsfeld-Mustopa Conjecture \cite[Conjecture 4.14]{ELM}.

More generally, we say a subspace $W\subset H^0(X,F)$ is generating if the evaluation map $\eval_{W}:W\otimes \cO_X\to F$ is surjective. In this case, the vector bundle $S_{F,W}:=\ker(\eval_W)$ is called the {\em generalized syzygy bundle} associated to $(F,W)$. An important particular case is when $F$ is an invertible sheaf $\mathcal O_X(L)$ and $W\subset |L|$ is a base-point-free linear system. This more general notion was first introduced by Brenner in \cite{B}, where he asked if there always exists a subspace $W\subset H^0(\cO_{\PP^n}(d))$ such that $S_{\cO_{\PP^n}(d),W}$ is stable. For an affirmative answer to this question the reader can look at \cite{C}, \cite{CMM} and \cite{MM}.

In this paper we study the moduli spaces of generalized syzygy bundles $S_{F,W}$; for readability, we write just $S$  when no confusion is likely to arise. We always assume $\dim X\ge 2$; the curve case is very different and it has been extensively studied (e.g., Butler  \cite{Bu} shows that the syzygy bundle of a line bundle of degree $d$ is stable if $d\ge 2g+1$).

We focus on simple vector bundles $F$ with fixed rank $r$ and Chern classes $c_i$; they are naturally an open subspace in the algebraic space of simple sheaves $\Spl(r;c_i)$. For any  integer $w$ we define the moduli space of generalized syzygy bundles $\cG^0_U$ to be the natural algebraic space structure on pairs $(F,W)$ where $F\in \Spl(r;c_i)$ is a simple vector bundle such that $H^1(F)=H^1(F^*)=0$ and $W\subset H^0(X,F)$ is a generating $w$-dimensional subspace. We show that each generalized syzygy bundle $S_{W,F}$ so obtained  is simple, of rank $r'=w-r$ and Chern classes $c_i'$ given by $\sum_{i}c_{i}'t^i =  (\sum _{i }c_{i}(F)t ^i)^{-1}$.

Our first result is that the set maps $(F,W)\mapsto F$ and $(F,W)\mapsto S_{F,W}$ are induced by morphisms $\pi:\cG^0_U\to \Spl(r;c_i)$ and $\alpha:\cG^0_U\to \Spl(r';c_i')$.

The morphism $\pi:\cG^0_U\to \Spl(r;c_i)$ is smooth and \'etale locally quasiprojective over its image $U$. In fact, it is locally open in a Grassmann bundle.

For every $\cL\subset \Spl(r;c_i)$ locally closed subspace, we let $\cG^0_{\cL}:=\pi^{-1}(\cL)$. Note that $\cG^0_{\cL}$ is smooth if and only if $\cL\cap U$ is smooth.

The main theorems of this paper are the following:

\begin{theorem}\label{embedding}
    The morphism $\alpha:\cG^0_U\to \Spl(r';c_i')$ is a locally closed embedding. 
\end{theorem}

We remark that the open subspace $U\subset \Spl(r;c_i)$ of points $[F]$ such that $F$ is a vector bundle  globally generated and $H^1(F)=H^1(F^*)=0$ can be assumed to be non-empty after sufficient twisting by an ample line bundle, as explained in Remark \ref{assumptions}.

The theorem implies that the restriction of $\alpha$ induces a locally closed embedding $\cG^0_{\cL}\to \Spl(r';c_i')$. 
In case $X$ is a $K3$ surface and $\cL$ is a Lagrangian in the symplectic space $\Spl(r;c_i)$, we show in \cite{FM} that the image of $\cG^0_{\cL}$ (if nonempty) is also a Lagrangian in $\Spl(r';c_i')$. 

\begin{theorem}\label{syz_open}
    Assume moreover that $H^1(X,\cO_X)=0$. Let $V$ be the open subspace of $U$ defined by $H^2(F^*)=0$. Then the morphism $\alpha:\cG^0_V\to \Spl(r';c_i')$ is an open embedding.
\end{theorem}

For this theorem to be meaningful, we want $V$ non-empty. This happens very rarely if $\dim X= 2$, but can be achieved by sufficient twisting if $\dim X\ge 3$. In particular, if $X$ is a Calabi-Yau threefold, then $U=V$.

While our statements are about algebraic spaces, our proofs heavily involve the corresponding moduli stacks, denoted respectively by $\SSpl(r;c_i)$ and $\cG^0_{\cU}$. We first  define the morphism $\bar\alpha: \cG^0_{\cU}\to \SSpl(r';c_i')$ by using the  existence of a universal family; the morphism $\alpha$ is then induced by the universal property of coarse moduli spaces.

To prove Theorem \ref{embedding} we first check that  $\alpha$ is injective on isomorphism classes of points and that its differential is also injective. Then we go back to algebraic stacks, we construct in Definition \ref{defSyz} a locally closed substack $\SSyz_{\cU}^{w}$ in $\SSpl(r';c_i')$ through which the morphism $\bar\alpha$ factors, and a morphism $\bar\gamma:\SSyz^{w}_{\cU}\to \cG^0_{\cU}$. Finally, we prove  that $\bar\alpha:\cG^0_{\cU}\to \SSyz^{w}_{\cU}$ is an isomorphism with $\bar\gamma$ as inverse.

Since $\alpha:\cG^0_V\to \Spl(r';c_i')$ is a locally closed embedding, to prove Theorem \ref{syz_open} it is enough to show that $\alpha $ is smooth at every point of $\cG^0_V$;  we demonstrate this by using the infinitesimal criterion of smoothness.

\vskip 4mm

We explain why we have our assumptions in the main theorems. We choose to work with simple bundles so that we have a moduli algebraic space. The condition $H^1(F^*)=0$ is required to assure that $S_{F,W}$ is simple while the condition $H^1(F)=0$ is used to construct the Grassmann bundle $\cG_U$, as it guarantees that $v=h^0(F)$ is locally constant on $U$.

The Example \ref{exp2} and the following lemmas show that the condition $H^1(X,\cO_X)=0$ is necessary in Theorem \ref{syz_open}. A possible way to construct enough sheaves  to fill an open subset in $\Spl(r';c_i')$ is discussed in Subsection \ref{very_gen}.

The condition $H^2(F^*)=0$ is  needed for the smoothness criterion to apply; we could weaken it by working with the open subspace of $\cG^0_U$ parametrizing pairs $(F,W)$ such that $H^2(F^*)\to W^*\otimes H^2(\cO_X)$ is injective. Note that if $H^2(\cO_X)=0$ this gives us back the condition $H^2(F^*)=0$.

Our assumptions are verified in a large number of significant cases. In the last section, we include many explicit examples.
Moreover, we  set up a recursive construction of open/locally closed subspaces of  $\Spl(r;c_i)$ (see Corollaries \ref{coro1} and \ref{coro2}, and Remark \ref{remcoro} that are smooth, rational, quasiprojective varieties.

\vskip 2mm

\noindent {\bf Acknowledgement.} The first author is thankful for hospitality at Universitat de Barcelona and Chennai Mathematical Institute, as well as to Sissa for a sabbatical year, during which this research took place.

\section{Notation and background material}

 We fix an  algebraically closed field  $\KK = \overline{\KK}$. All schemes/algebraic spaces/stacks will be locally of finite type over $\KK$ and points will always be $ \KK$-valued points.  We also fix a reduced connected Gorenstein projective variety $X$ of pure dimension $n\ge 2$ and we denote by $\omega _X$ its canonical line bundle.  

We denote $\Gm$ as $\Spec \KK[t,t^{-1}]$ with the group structure defined by multiplication.

We will frequently fix the Chern classes of a vector bundle or a coherent sheaf $F$ on $X$. If $\KK$ is the field of complex numbers, we can think of $c_i(F)$ as elements in $H^{2i}(X^{an})$. In general, we will view them as degree $-i$ operators on Chow groups on $X$ modulo algebraic (not just rational) equivalence, so that Chern classes are locally constant in the fibres of a flat family.

\subsection{Moduli spaces and moduli stacks of simple bundles}

In this section, we collect for the reader's convenience the definitions and results that we will use later and we also fix relevant notation.

A coherent sheaf $E$ on $X$ is {\em simple} if the natural injection $$H^0(X,\cO_X) \to \End_{\cO_X}(E)$$ is an isomorphism. Thus, $E$ being simple is equivalent to $\End_{\cO_X}(E)=\KK\cdot \id_E$.

\begin{definition} A {\em family of simple sheaves} on $X$ {\em parameterized by} a scheme $B$ is a coherent sheaf $\cF$ on $X\times B$ such that
\begin{itemize}
\item[(1)] $\cF$ is flat over $B$, and
\item[(2)] For any $b\in B$, $\cF_b:=\cF_{|X\times \{b\}}$ is a simple sheaf.
\end{itemize}
Isomorphisms of families $\cF_1$ and $\cF_2$ are isomorphisms of sheaves on $X\times B$.
\end{definition}

The pullback of a family $\cF$ parameterized by $B$ via a morphism $\varphi :B_1\to B$ is $(id_X\times \varphi)^*\cF.$ 

\begin{theorem}  The pseudofunctor $(Sch)^{op}\to (Groupoids)$ associating to a scheme $B$ the groupoid of families of simple sheaves parametrized by $B$ is an algebraic stack $\cS pl(X)$.
\end{theorem}

\begin{proof} This is a special case of Theorem 4.6.2.1 in \cite{L-MB}.
\end{proof}

\begin{remark}\label{open_in_moduli}
    By definition, an open substack $\cU$ of $\SSpl(X)$  is determined by its points, i.e., the set of (isomorphism classes of) coherent sheaves $\cF$ on $X$ such that $\cF\in \cU$. To show that a set of isomorphism classes $\cU$ defines an open substack is equivalent to showing that, for any family of sheaves parameterized by a scheme $B$, the set \[
    B_{\cU}:=\{b\in B\,|\, \cF_b\in \cU\}
    \]
    is open in $B$.
\end{remark}
The stack $\cS pl(X)$ has infinitely many connected components. If we fix the rank $r$ and the Chern classes $c_i$, we can consider the substack of families $\cF$ such that for any $b\in B$, $\cF _b$ has $rank(\cF_b)=r$ and $c_i(\cF_b)=c_i$. The stack $\cS pl(r;c_i)$ so obtained is open and closed in $\cS pl(X)$. 
\newcommand{\eps}{\varepsilon}

\begin{definition}
     A morphism of algebraic stacks $\eps:G\to X$, with $X$ an algebraic space, is called a $\Gm$-gerbe if and only if, (\'etale) locally on $X$, $G$ is isomorphic to a product $X\times B\Gm$, where $B\Gm$ is the stack quotient of a point by the trivial action of $\Gm$.
\end{definition}

\begin{proposition}
The stack $\cS pl(r;c_i)$ has a coarse moduli space $\Spl(r;c_i)$ which is a (possibly non separated) algebraic space. The structure morphism $$\eps:\SSpl(r;c_i)\to \Spl(r;c_i) $$
is a $\Gm$-gerbe.
\end{proposition}
\begin{proof} This is a direct application of Theorem 5.1.5 in \cite{ACV} in the case where $\mathbb S=\Spec \KK$, the stack $\mathcal X=\cS pl(r;c_i)$ and the group scheme $H$ is the multiplicative group $\KK^*$ (or $\mathbb G_m(\KK)$). We let $\Spl(r;c_i):=\mathcal X^H$. By \cite[Theorem 5.1.5]{ACV} (3), each $\KK$-point of $\Spl(r;c_i)$ has trivial automorphism group. Therefore, $\Spl(r;c_i)$ is an algebraic space and thus the coarse space of $\SSpl(r;c_i)$.

This agrees with the stack-free definition of $\Spl(r;c_i)$ (see, for instance, \cite{Mu}) because of the universal property 2 in \cite[Theorem 5.1.5]{ACV}.

The rest follows by the identification of $\mathcal X^H$ with $[\mathcal X/BH]$.
\end{proof}

\begin{remark}\label{gerbe}
 Let $G$ be an algebraic stack, $X$ an algebraic space, and $\eps:G\to X$ a morphism that makes $G$ into a $B\Gm$-gerbe over $X$. Then we have:
\begin{enumerate}
    \item the pullback functor $\eps^*:\Coh(X)\to \Coh(G)$ is fully faithful; 
    \item $\eps^{-1}$ induces a bijection from points on $X$ to isomorphism classes of points on $G$;
    \item $\eps^{-1}$ induces a bijection from open subspaces of $X$ to open substacks of $G$;
    \item for every coherent sheaf $\cF$ on $G$, and every point $x:\Spec \KK\to G$, the fiber $x^*(\cF)$ has a canonical $\Gm$ action;
    \item $\cF\in\Coh(G)$ is (isomorphic to) the pullback of a coherent sheaf on $X$ if and only if, for every $x$ in $G$, the fiber $x^*(\cF)$ is $\Gm$-invariant. 
\end{enumerate}
\end{remark}

By definition, there is a universal family $\cF_{un}\in Coh(X\times \cS pl(r;c_i))$ parameterized by $\cS pl(r;c_i)$ such that for any scheme $B$, the natural functor \[
Mor(B,\cS pl(r;c_i))\to \{\text{families parameterized by $B$}\}
\]
defined by $g\mapsto (\id_X\times g)^*\cF_{un}$ is an equivalence of groupoids.

We define an equivalence relation among families of coherent sheaves parameterized by a scheme $B$ by saying that $\cF$ is equivalent to $\cG$ if and only if there is a line bundle $L$ on $B$ such that $\cG$ is isomorphic to $\cF\otimes p_B^*L$ where $p_B:X\times B\to B$ is the natural projection. Notice that this implies that for every point $b$ in $B$, the fibers $\cF_b$ and $\cG_b$ are isomorphic, and that isomorphic families are equivalent but not, in general, vice versa.

We say that  a family $\cF$ parameterized by $Spl(r;c_i)$ is universal if, for every scheme $B$, the natural map \[
Mor(B,Spl(r;c_i))\to \{\text{equivalence classes of families parameterized by $B$}\}
\]
defined by $g\mapsto (\id_X\times g)^*\cF$ is a bijection.

If a universal family parameterized by $Spl(r;c_i)$ exists, then it is unique up to tensoring with the pullback of a line bundle on $Spl(r;c_i)$. However, in general, it does not exist.

In this paper, we will use the universal family $\cF_{un}$ on $\cS pl(r;c_i)$ to define our moduli of syzygy bundles. Alternatively, we could work \'etale locally on $Spl(r;c_i)$ (because \'etale locally a universal family exists) and then use \'etale descent, but we find this method more cumbersome (and it is in any case equivalent).

\subsection{Open and locally closed conditions}

In this subsection, we will define open and locally closed  substacks of $\SSpl(r;c_i)$ and therefore open and locally closed algebraic subspaces of $\Spl(r;c_i)$ that we will be using repeatedly in the rest of the paper.

\begin{notation}
    In this subsection, we will let $\cF\in \Coh(X\times B)$ be a family of simple sheaves on $X$ of rank $r$ and Chern classes $c_i$ parameterized by a scheme $B$; let $u:B_1\to B$ be a morphism and $\cF_1=(\id _X\times u)^*\cF$ the pullback family parameterized by $B_1$. Let $p_B:X\times B \to B$, $p_{B_1}:X\times B_1\to B_1$  and $p_X:X\times B\to X$ (and also $X\times B_1\to X$) be the projections.
\end{notation}

We note that if $\cF$ is locally free, then $\cF^*$ is also a locally free family parameterized by $B$, and its pullback to $B_1$ is naturally isomorphic to $\cF_1^*$.

\begin{remark}\label{openconditions}
    Each of the following conditions on points $y\in B$  defines an open subset $U$ in $B$:\begin{enumerate}
        \item the sheaf $\cF_y$ is locally free;
        \item the group $H^i(X,\cF_y)$ vanishes, for a fixed $i\in \{0,\ldots,n\}$.
    \end{enumerate} 
\end{remark}

\begin{corollary}
The inverse image of $U$ in $B_1$ is exactly the set of points $y_1\in B_1$ where $\cF_1$ satisfies the corresponding condition. Thus each condition defines an open substack in $\SSpl(r;c_i)$, and  an open algebraic subspace in $\Spl(r;c_i)$.
\end{corollary}

\begin{proof}
    It immediately follows from the fact that these conditions are defined on points.
\end{proof}
\begin{lemma}
Assume that for every $y\in B$, the morphism $$\varphi_0(y):p_{B*}(\cF)\otimes k(y)\to H^0(X,\cF_y)$$ is surjective. Then $p_{B*}\cF$ is locally free, and the natural morphism $$\theta:u^*p_{B*}\cF \to p_{B_1*}\cF_1$$  is an isomorphism.
\end{lemma}
\begin{proof}
    That $p_{B*}\cF$ is locally free follows from cohomology and base change, since $\varphi_{-1}(y)$ is always an isomorphism.
    Therefore, $\theta$ is a homomorphism between locally free sheaves that is an isomorphism on every fibre and it follows that it is an isomorphism.
\end{proof}

\begin{proposition}
    Assume that $\cF$ is locally free. The following conditions  are equivalent:
    \begin{enumerate}
        \item $p_{B*}\cF$ is locally free of rank $w$, and for every $y\in B$ the morphism $\varphi_0(y)$ is surjective;
        \item $R^np_{B*}(\cF^*\otimes p_X^*\omega_X)$ is locally free of rank $w$.
    \end{enumerate}
    Moreover, if they are true for $\cF$ over $B$, they are also true for $\cF_1$ over $B_1$.
\end{proposition}
\begin{proof}
    This follows from Grothendieck-Serre duality.
\end{proof}

\begin{remark}
    The natural morphism $u^*R^np_{B*}\cF\to R^np_{{B_1}*}\cF_1$ is always an isomorphism, because $H^{n+1}(X,\cF_y)=0$ for every $y\in B$. Hence it defines a coherent sheaf on the moduli stack $\SSpl(r;c_i')$. The same argument shows that $R^n(\cF^*\otimes p_X^*\omega_X)$ induces a coherent sheaf  over the open substack in $\SSpl(r;c_i)$  where $\cF$ is locally free.
\end{remark}

\vskip 2mm
\subsection{Infinitesimal criterion for smoothness} In this short subsection we recall  a useful {\em infinitesimal criterion for smoothness} that we will use in the proof of our main results.
We will use notation from Chapter 2 of \cite{Se}, to which we refer for details.

\begin{proposition} \label{criterion_smoothness}
  Let $f:X\to Y$ be a morphism of algebraic spaces (or algebraic stacks) and $p\in X$ a point. If   the morphism $f$ induces a surjective map $df: T_pX\to T_{f(p)} Y$ on tangent spaces and an injective map $o(f)$ on obstruction spaces, then $f$ is smooth at $p$.  
\end{proposition}
\begin{proof}
    This is a classical result, so we include an outline of proof for the reader's convenience. If $X$ and $Y$ are both affine this is  \cite[Theorem 2.1.5]{Se}.
    
    For the general case, since everything is locally of finite type over $\KK$, smoothness is equivalent to formal smoothness, thus a lifting property for small extensions. In particular, smoothness can be checked locally in the smooth topology on $Y$, i.e., if $\tilde f:\tilde X\to \tilde Y$ is obtained from $f$ by a base change $\tilde Y\to Y$, and $\tilde p\in \tilde X$ maps to $p$, then $\tilde f$ is smooth in $\tilde p$ if and only if $f$ is smooth in $p$. On the other hand, the obstruction spaces do not change under smooth base change, and the tangent spaces change in compatible ways so that $df$ is surjective at $p$ if and only if $d\tilde f$ is surjective at $\tilde p$.
    
    We can therefore assume that $Y$ is an affine scheme, since by definition $Y$ has a smooth atlas of affine schemes. For the same reason, we can find an affine scheme $S$, a point $s\in S$ and a smooth morphism $\pi:S\to X$ such that $\pi(s)=p$.
    For any small extension, a lifting exists from $Y$ to $X$ if and only if it exists from $Y$ to $S$. Again the obstruction space for $X$ at $p$ is also an obstruction space for $S$ at $s$; the differential $d(f\circ \pi):T_sS\to T_{f(p)}Y$ factors via $T_pX$, which means that $d(f\circ \pi)$ is surjective at $s$ if and only if $df$ is surjective at $p$, since $d\pi$ is surjective at every point by the smoothness of $\pi$.
\end{proof}


\subsection{Infinitesimal deformations of sheaves and their homomorphisms}

For the sake of completeness we start this subsection describing the tangent spaces and the obstruction spaces for $\SSpl(r;c_i)$ and $\Spl(r;c_i)$.

\begin{proposition}\label{tgt-obs} Let $F$ be a simple sheaf on $X$, of rank $r$ and Chern classes $c_i$. It holds:
\begin{itemize}
\item[(1)] $T_{[F]}\SSpl(r;c_i)\cong T_{[F]}\Spl(r;c_i)\cong Ext^1(F,F)$. In particular, if $F$ is a vector bundle then $T_{[F]}\Spl(r;c_i)\cong H^1(F^*\otimes F)$.
\item[(2)] An obstruction space for $\SSpl(r;c_i)$ (resp. $\Spl(r;c_i)$) at $[F]$ is $Ext^2(F,F).$  If $F$ is a vector bundle then $Ext^2(F,F)\cong H^2(F^*\otimes F)$.
\end{itemize}
\end{proposition}
\begin{proof} For the case of $\Spl(r;c_i)$ see, for instance,   \cite[Example 6.3.7(2)]{FG}. On the other hand, the structure morphism from $\SSpl(r;c_i)$ to $\Spl(r;c_i)$ is smooth, hence induces an isomorphism on obstructions, and a $B{\mathbb G}_{\text m}$ gerbe, thus it induces an isomorphism on tangent spaces.
\end{proof}

\begin{remark} If $Ext^2(F,F)=0$, then the algebraic space
$Spl(X)$ and the algebraic stack $\cS pl(X)$ are smooth but not, in general, conversely. For instance, if $X$ is a smooth projective complex K3 surface, then for any  simple sheaf $F$ on $X$ we have $Ext^2(F,F)=Hom(F,F)=\CC$ but the moduli space $\Spl(r;c_i)$ of simple sheaves is smooth (see, for instance, \cite[Theorem 0.1]{Mu}).
\end{remark}

\begin{proposition}\label{tgt-quot} Let $Y$ be a projective scheme, $E\in Coh(Y)$
and $\eta=[E\twoheadrightarrow F]\in Quot_Y^p(E)$. Then,
$$ T_{[\eta ]}Quot_Y^p(E)\cong Hom(S,F)$$
where $S=ker(E\to F)$ and an obstruction space is $Ext^1(S,F)$.
\end{proposition}
\begin{proof} See, for instance,   \cite[Theorem 6.4.9]{FG}.
\end{proof}

In particular, if $Y=Spec \, \KK $, $V$ is a $\KK $-vector space of dimension $v$, $w\le v$ a positive integer and $W\subset V$ is a subspace of dimension $w$, then
\begin{equation}\label{tg-grass}
T_{[W]}Gr(w,V)\cong Hom_{\KK-sp}(W,V/W)
\end{equation}
and $Gr(w,V)$ is smooth of dimension $w(v-w)$.

\begin{corollary}\label{def-triple}
    Let $Y$ be a projective scheme, and $\phi:E\to F$ a surjection of vector bundles on $Y$. Let $D$ be the deformation functor of the triple $(E,F,\phi)$. Then the tangent space to automorphisms, tangent space, and an obstruction space for $D$ can be computed as $T^iD=\HH^i(A)$ for $i=0,1, 2$ where $A\in C^{0,1}(\Coh(Y))$ is the complex defined by $A^0:=E^*\otimes E\oplus F^*\otimes F$, $A^1:=E^*\otimes F$, and $d(\alpha,\beta):=\phi\circ \alpha-\beta\circ\phi$.
\end{corollary}
\begin{proof}
This is well known, we include a sketch of the argument for completeness.
We have an exact sequence of complexes \[
0\to A' \to A \to E^*\otimes E\to 0
\]
where $A'$ is the subcomplex $[F^*\otimes F\to E^*\otimes F]$. We consider the forgetful morphism $D\to D_E$, where we only remember the bundle $E$. It induces a long exact sequence in hypercohomology. 

Let $S:=\ker \phi$;
the exact sequence \[0\to F^*\otimes F\to E^*\otimes F \to S^*\otimes F\to 0 \]
induces a morphism of complexes $A'\to S^*\otimes F[-1]$ which is an isomorphism on cohomology bundles. Hence it induces an isomorphism in hypercohomology;
thus $\HH^i(A')$ is naturally isomorphic to $H^{i-1}(S^*\otimes F)$. The fiber of $D\to D_E$ are deformations of $\phi$ and $F$ over a given deformation of $E$, which is the relative Quot scheme. Thus the result follows from Propositions \ref{tgt-obs} and \ref{tgt-quot}.
\end{proof}


\subsection{Definition of generalized syzygy bundles}

Let $F$ be a globally generated vector bundle of rank $r$ on $X$. 
Fix an integer $w\ge n+r$ and let $W\subset  H^0(X, F)$ be a subspace of dimension $w$. We have an induced evaluation map:
 $$eval_{W}: W\otimes \mathcal O_X \to F.$$
When $eval_W $ is surjective 
 we define the {\em generalized syzygy bundle} $S$ on $X$ as its kernel. Therefore, $ S$  fits inside a short exact sequence:
\begin{equation}
    \label{exact1}
e: \quad 0\to  S\to W\otimes \mathcal O_X \to  F\to 0, \text{ and }
\end{equation}
the rank and Chern classes of $S$ are given by:
$$\begin{array}{rcl} r':=rank(S)& = & \dim W- rank(F)= w-r, \\
c_1':=c_1(S) & = & -c_1(F) \text{ and}\\
\sum_{i}c_{i}(S)t^i& = & (\sum _{i }c_{i}(F)t ^i)^{-1}
\end{array}
$$
We will call  $S$ the {\em generalized syzygy bundle} associated to the couple $(F,W)$.

\begin{notation}
    In the rest of the paper we fix the rank $r$ and the Chern classes $c_i$ for $F$ and we write $r'$ and $c_i'$ for the rank and the Chern classes of the associated syzygy bundle.
\end{notation}
Notice that when $F=\cO _X(L)$ is a very ample line bundle on $X$ and $W=H^0(X,L)$ we recover the classical definition of syzygy bundle, namely, the {\em syzygy bundle} $M_L$ associated to $\cO _X(L)$ is by definition the kernel of the evaluation map  $eval_{L}: H^0{\mathcal O}_X(L)\otimes \mathcal O_X \to {\mathcal O}_X(L)$. Thus, $M_L$ sits in the exact sequence:
\[
0\to M_L\to H^0({\mathcal O}_X(L)\otimes \mathcal O_X \to {\mathcal O}_X(L)\to 0.
\]
Syzygy bundles and generalized syzygy bundles arise in a variety of geometric and algebraic problems and there has been great interest in trying to establish their stability with respect to a suitable ample line bundle as well as in describing their moduli spaces (see, for instance, \cite{BP}, \cite{B}, \cite{CL}, \cite{CMM}, \cite{ELM}, \cite{MM}, \cite{MS}, \cite{MS2} and \cite{MR}).

When $X$ is a smooth curve of genus $g\ge 1$, the situation is well understood thanks to work of  Butler \cite{Bu}, Beauville \cite{Be} and Ein-Lazarsfeld \cite{EL}. In particular, $M_L$ is stable provided $\deg (L)\ge 2g+1$.

In higher dimension the stability of $M_L $  is equivalent to the stability of the pullback $\varphi ^*_{|L|}(T_{\PP })$ of the tangent bundle of $\PP:=\PP (H^0(X,L)^*)$ where $\varphi_{|L|}:X\to \PP(H^0(X,L)^*)$ is the morphism associated to the complete linear system $|L|$.


\section{Definition of the morphism $\alpha$}
In  this section, we gather the basic results on generalized syzygy bundles  needed in the sequel and we explicitly construct the morphism $\alpha:\cG^0_U\to \Spl(r';c_i')$.
We start with an easy lemma.

\begin{lemma} \label{aux} Let $F $ be a simple vector bundle on a reduced connected projective scheme. If $F $ is generated by global sections and $H^0(F)=0$, then $H^0(F^*)=0$ unless $F=\cO_X$.
\end{lemma}

\begin{proof} Assume $H^0(X,F^*)\ne 0$. We take non trivial sections $s\in H^0(X,F)$ and $t\in H^0(X,F^*)$. Then $s\otimes t\in H^0(X,F\otimes F^*)$ is a nonzero endomorphism of $F$ which in each fibre has rank at most one.
If $\rk F\ge 2$ this implies that $s\otimes t$ cannot be a scalar. If $F$ is a line bundle and $s\otimes t$ is a (non zero) scalar, then $s$ can never vanish and thus $F$ is trivial.
\end{proof}

\begin{proposition}\label{open}  The locus in  $\cS pl(r;c_i)$ (rsp. $Spl(r;c_i)$) of points $[F]$ such that 
 $F$ is a vector bundle satisfying $H^1(F)=H^1(F^*)=0$ and $F$ is generated by global sections defines a (possible empty) open substack $\cU$ (resp. algebraic subspace $U$).
 
 The set-theoretic function $v$ associating to each point $[F] $ of this open locus the dimension of $H^0(F)$ is locally constant. Hence it is constant on every connected component of $U$ and $\cU$.
\end{proposition}

\begin{proof}
     By Remark \ref{open_in_moduli}, to prove that it is an open substack we only need to show that, for any family $\cF$ of simple sheaves parameterized by a scheme $B$, the locus \[\{b\in B\,|\, \cF_b \text{\ has the required properties}\}\]
    is open in $B$.
    
    The condition of being a vector bundle is open; once we work with vector bundles, also the dual family is flat. The vanishing of any fixed cohomology group is an open condition by Remark \ref{openconditions}(2).\\
    Because $H^1(\cF_b)=0$ for every $b$, again by cohomology and base change it follows that  the sheaf $q_*(\cF)$ (where $q:X\times B\to B$ is the projection) is a vector bundle with fibres $H^0(X,\cF_b)$. A vector bundle has locally constant rank, so the function $v$ is locally constant. Moreover, the locus where $\cF_b$ is generated by global sections is the complement of the support of the cokernel of the evaluation map $q^*q_*\cF\to\cF$.
    
    The openness in $\Spl(r;c_i)$ follows from Remark \ref{gerbe} (3).
\end{proof}

\begin{definition}\label{def-U}
      From now on, we will define $\cU$ (resp. $U$) to be the  open   substack (resp. algebraic subspace) of $\SSpl(r;c_i)$ (resp. $\Spl(r;c_i)$) defined in Proposition \ref{open}. In the rest of the paper we assume that $\cU$ is nonempty. 
\end{definition}

The assumption that $\cU$ be non empty is not very restrictive once one has that the vector bundle locus is nonempty. By Serre theorem it can always be achieved by twisting by a sufficiently ample line bundle (because $n\ge 2$); twisting preserves simplicity and induces isomorphisms between different moduli spaces (and stacks) of simple sheaves (see Remark \ref{assumptions}).

\vskip 2mm
In this paper, we will prove that generalized syzygy bundles associated to vector bundles $F$ in $U$ are always simple and we will study  their locus inside $Spl(r';c_{i}')$. To this end, we need to fix some more notation. 

\begin{definition}
We  define
$$\cG _\cU:=\{(F,W) \mid F\in \cU, W\subset H^0(F), \dim W=w\};$$
 it is the relative Grassmannian of subspaces of dimension $w$ in the vector bundle $\tilde{p}_*(\cF_{\cU})$ on $\cU$, where $\tilde{p}: \cU\times X\to  \cU$ is the projection and $\cF_{\cU}$ the universal sheaf on $X\times \cU$.
The natural projection $\pi: \cG_\cU\to \cU$ is a Grassmann bundle locally in the smooth topology on $\cU$.

Moreover, we define 
$$\cG^0 _\cU:=\{(F,W)\in \cG _\cU  \mid eval_W \text{ is surjective}\}.$$
It is an open substack of $\cG _\cU$.
\end{definition}

Let us call $p:  \cG^0 _\cU \times X \to  \cG^0 _\cU$ the natural projections and  $\tilde\pi:X\times \cG^0 _\cU\to X\times \cU$  the map induced  by $\pi $. Let $\cF_{\cU}$ be the universal bundle on 
$X\times \cU$, and let $\cF:=\tilde \pi^*\cF_{\cU}$ be the pullback vector bundle on $\cG^0_{\cU}\times X$. The pushforward $\tilde{p}_*(\cF_{\cU})$ is a vector bundle of rank $v$ on $\cU$
 such that the fiber over $[F]$ is canonically isomorphic to $H^0(F)$. This implies that the canonical morphism  $\pi^*\tilde{p}_{*}\cF_{\cU}\cong p_{*}\cF$  is an isomorphism.
 
By definition of Grassmann bundle, on $\cG^0 _\cU$ (induced by $\cG_{\cU}$) we have the rank $w$ universal subbundle $\cW \subset p_{*}\cF$ such that the fiber on $[(F,W)]$ is $W\subset H^0(F)$; the morphism $p^*\cW\to \cF$ induced by adjunction from the inclusion is surjective on $\cG^0_{\cU}\times X$ by definition, and thus defines a vector bundle $\cS$ on $X\times \cG^0 _\cU$ via the exact sequence of vector bundles 
 \begin{equation} \label{univ_syz}
0\to \cS \to p^*\cW \to \cF\to 0. 
\end{equation}

\vskip 2mm

Note that this construction only works on $\cU$ and not on $U$. In fact, we do get an induced Grassmann bundle $\pi:\cG_U \to U$, because subspaces are invariant under scalar multiplication, but not the exact sequence, because the action of $\Gm$ on these vector bundles is non-trivial.

\subsection{Simplicity and cohomological calculations}

We want to use the family $\cS$ to define a morphism $\bar \alpha$ from $\cG_{\cU}^0$ to $\SSpl(r';c_i')$. To  this end, we first need to check that the fibers of $\cS$ are simple bundles.
 
\begin{proposition} \label{simple}

 Let $X$ be a  projective variety with
$dim X=n \ge 2$ and let $S$ be a generalized syzygy bundle associated to $(F,W)\in \cG_\cU^0$. Then, it holds:
\begin{itemize}
    \item[(a)] $H^0(S)=0$,
    \item[(b)] There is a natural isomorphism $W^*\to H^0(S^*)$,
    \item[(c)] $S$ is simple.
    \end{itemize} 
\end{proposition}

\begin{proof} (a) We consider the exact sequence
\begin{equation}
    \label{exact1bis}
e: \quad 0\to S\to W\otimes \mathcal O_X \to  F\to 0, 
\end{equation}
and  taking cohomology we get
$$
0\to H^0(S)\to W \to H^0(F)
$$
and we conclude that $H^0(S)=0$.

\vskip 2mm
(b) We dualize the exact sequence (\ref{exact1bis})
and we get
\begin{equation}
    \label{exact2}
0\to  F^*\to W^*\otimes \mathcal O_X \to  S^*\to 0.
\end{equation}
We take cohomology and we obtain
$$
0\to H^0(F^*)\to W^*\to H^0(S^*)\to H^1(F^*)\to \cdots
$$
By Lemma \ref{aux}, $H^0(F^*)=0$
and by hypothesis $H^1(F^*)=0$. Therefore there is a natural isomorphism  $W^*\to H^0(S^*)$.

\vskip 2mm

(c) We tensor with $S$ the exact sequence (\ref{exact2}) and we get the exact sequence
\begin{equation}
    \label{exact3}
0\to F^*\otimes S\to W^*\otimes S \to S^*\otimes S \to 0.
\end{equation}
Since $H^0(X,S)=0$, we deduce that $H^0(S^*\otimes S)\subset H^1(F^*\otimes S).$ Tensoring with $F^*$ the exact sequence (\ref{exact1bis}), we get
\begin{equation}
    \label{exact4}
0\to  F^*\otimes S\to W\otimes F^* \to F ^*\otimes F \to 0
\end{equation}
and the exact cohomology sequence
$$
0\to H^0(F^*\otimes F ) \to H^1(F^*\otimes S) \to W\otimes H^1(F^*).
$$
By hypothesis, $F$ is simple and $H^1(F^*)=0$ which allows us to conclude that $\dim H^1(F^*\otimes S)= 1$ and $H^0(S \otimes S^*)=\KK$. Therefore, $S$ is simple.
\end{proof}

\begin{remark}
    Note that the condition $H^1(F)=0$ that we imposed on $\cU$ is not required to prove that $S$ is simple. \end{remark}

\begin{definition}
    The fibers of the family of vector bundle $\cS$ parameterized by $\cG^0_{\cU}$ in the exact sequence (\ref{univ_syz}) are simple, have rank $r'$ and Chern classes $c_i'$. Therefore $\cS$ defines a morphism $$\bar \alpha:\cG^0_{\cU}\to \SSpl(r';c_i')$$ which extends the set theoretic map  $(F,W)\mapsto S$.
\end{definition}
 The morphism $\bar \alpha$ induces a morphism $\alpha$  of algebraic spaces from $\cG^0_U$ to $\Spl(r';c_i')$. The commutative diagram
\[
\xymatrix{\cG^0_\cU \ar[r]^{\bar \alpha \,\, \,\ }\ar[d] &\SSpl(r';c_i') \ar[d]\\
\cG^0_U \ar[r]^{\alpha \,\,\,\ } &\Spl(r';c_i')}
\]

\noindent is cartesian and the vertical maps are $\Gm$-gerbes. Therefore, the following properties of $\alpha$ and $\overline{\alpha }$ imply each other; is injective, has injective differential, is a locally closed embedding, is an open embedding.
    
\section{Proof of  Theorem \ref{embedding}}

We will first prove that $\alpha$ is injective on isomorphism classes of points and that its differential is injective. Then we construct a locally closed substack $\SSyz_{\cU}^{w}$ in $\SSpl(r';c_i')$ through which $\bar \alpha$ factors, and a morphism $\bar \gamma:\SSyz^{w}_{\cU}\to \cG^0_{\cU}$. Finally, we prove  that $\bar \alpha:\cG^0_{\cU}\to \SSyz^{w}_{\cU}$ is an isomorphism with $\bar \gamma$ as inverse. From this we deduce  the analogous statement for the moduli spaces.

\subsection{Injectivity of $\alpha$ and its differential}

\begin{proposition} \label{inj}  The morphism
$$\begin{array}{rcl}  \alpha: \cG_U^0 &  \to  & \Spl(r';c_{i}') \\
(F ,W) & \mapsto & S:= ker(  W\otimes \cO_X \twoheadrightarrow F)
\end{array}
$$ 
 is injective.
\end{proposition}

\begin{proof} 
By Lemma \ref{aux} we have $H^0(F^*)=H^n(F)^*=0$ and by hypothesis 
$H^1(F^*)=0$. So, we recover $W$ as $H^0(S^*)^*$ because the exact sequence (\ref{exact2})
gives us
$$
0=H^0(F^*)\to W^*\to H^0(S^*)\to H^1(F^*)=0.
$$
We define a vector bundle $H^*$ as the kernel of the evaluation map $H^0(S^*)\otimes \cO_X\to S^*$:

$$
0\to H^*\to H^0(S^*)\otimes \cO_X \to S^*\to 0.
$$
By dualizing this sequence and comparing it with (\ref{exact1}), we get that $H$ and $F$ are isomorphic. This isomorphism is unique up to scalar because $F$ is simple  and, moreover, scalars do not change subspaces. More precisely, we have got that $F^*$ is the syzygy bundle associated to $(S^*,H^0(S^*))$.
\end{proof}

\begin{proposition}\label{dif_inj} The differential map
$$d\alpha :   T_{[(F , W)]}\cG_U^0  \to T_{[S]}(\Spl(r';c_{i}')) 
$$
is injective.
\end{proposition}

\begin{proof} By  Proposition \ref{tgt-obs}, we have $$T_{[S]}(\Spl(r';c_{i}')) \cong H^1(S\otimes S^*)$$ and, by Proposition \ref{tgt-quot}, we have $$T_{[(F , W)]}\cG_U^0 \cong Hom(S,F)/T(Aut(W))\cong H^0(S^*\otimes F)/W\otimes W^*.$$ Now, the result immediately follows from  the exact cohomology sequence associated to the exact sequence
\begin{equation}\label{tensor}
0\to S\otimes S^*\to W\otimes S^*\to F\otimes S^*\to 0
\end{equation}
and the fact that $H^0(S^*)=W^*$.
\end{proof}

\subsection{The moduli stack and the moduli space of generalized syzygy bundles}

In this subsection, we will define a locally closed substack $\SSyz^{w}_{\cU}$ of $\SSpl(r';c_i')$ such that the morphism $\bar \alpha$ factors through $\SSyz^{w}_{\cU}$; and we will construct a morphism $\bar \gamma:\SSyz^{w}_{\cU}\to \cG^0_{\cU}$. In the next subsection we will prove Theorem \ref{embedding} by showing that  $\bar \alpha:\cG^0_{\cU}\to \SSyz^{w}_{\cU}$ is an isomorphism with $\bar \gamma$ as inverse.
We will use the structure of the proof of injectivity of $\alpha$, extending it from individual vector bundles to families. 

\vskip 2mm
We recall the set-up. 
Let $\cF_{\cU}$ be the universal bundle on 
$X\times \cU$, and let $\cF:=\tilde \pi^*\cF_{\cU}$ the pullback vector bundle on $\cG^0_{\cU}\times X$.
By definition, on $\cG_{\cU}^0\times X$ we have a canonical rank $w$ subbundle $\cW$ of the rank $v$ bundle $p_*\cF$, where $\tilde{\pi}:\cG^0_{\cU}\times X\to \cU\times X$ and $p:\cG^0_{\cU}\times X\to \cG^0_{\cU}$ is the projection. The induced morphism $p^*\cW\to \cF$ can be completed to an exact sequence \[
0\to \cS\to p^*\cW\to \cF\to 0
\] 
 where each restriction bundle $S$ of $\cS$ to a point times $X$ is a simple vector bundle with given rank $r'$, Chern classes $c_i'$ and $H^0(S)=0$. Moreover, dualizing yields an exact sequence \[
 0\to \cF^*\to p^*\cW^* \to \cS^*\to 0.
 \]
For every point $(F,W)$ in $\cG^0_{\cU}$ we have $H^0(F^*)=H^1(F^*)=0$. Hence the induced morphism $\cW^*\to p_*\cS^*$ is an isomorphism. Moreover, the induced morphism from  the fiber of $\cW$, $W^*$, to $H^0(S^*)$ is an isomorphism. This implies that $\varphi_0$ for $p_*\cS^*$ is surjective at every $(F,W)\in \cG^0_{\cU}$. In particular, $R^np_*(\cS\otimes p_X^*\omega_X)$ is locally free of rank $w$.

We will now define the locally closed substack $\SSyz^{w}_{\cU}$ in $\SSpl(r';c_i')$ in several steps. We start by defining a locally closed substack $\cZ$, and then proceed to select inside it open substacks.

    Let $\cZ$ be the substack  in $\SSpl(r';c_i')$ defined by families $(B,\cS)$ such that $\cS$ is locally free, $h^0(X,\cS_b)=0$ for every $b\in B$, and  $R^np_{B*}(\cS\otimes p_X^*\omega_X)$  is locally free of rank $w$.   By the universal property of the Fitting ideal filtration, $\cZ$ is locally closed in $\SSpl(r';c_i')$ (see Lemma 31.9.6(3) and Lemma 31.9.1  in \cite{St-Pr}).
    
    On $\cZ\times X$ the restriction of the universal sheaf is a vector bundle $\cSS$. We denote by $p_{\cZ}:\cZ\times X\to \cZ$ the projection, and by $\cW^*:=p_{\cZ*}(\cSS^*)$. By our assumption $\cW^*$ is a rank $w$ vector bundle and for each $[S]\in \cZ$ the natural map $\cW^*\otimes \kappa([S])\to H^0(S^*)$ is an isomorphism.

    We let $\cZ_1$ be the open substack in $\cZ$ where the evaluation homomorphism $p_{\cZ}^*\cW^*\to \cSS^*$ is surjective. On $\cZ_1\times X$ we define a locally free sheaf $\cH$ by the exact sequence \[
    0\to \cH^* \to p_{\cZ}^*\cW^*\to \cSS^*\to 0.
    \]
    $\cH$ is a flat family of vector bundles on $\cZ_1$ of rank $r$ and Chern classes $c_i$. We look at the open locus $\cZ_2\subset \cZ_1$ of points $[S]$ such that $\cH|_{[S]\times X}$ is simple. The image of $\bar \alpha$ is contained in $\cZ_2$.
     Therefore, the vector bundle $\cH$ on $\cZ_2\times X$ defines a morphism $\bar \beta:\cZ_2\to \SSpl(r;c_i)$. By Proposition \ref{inj}, we have that for any $[F]\in U\subset \Spl(r;c_i)$, the set theoretic image of  the fiber of $\cG^0_U$ over $[F]$ via $\beta\circ
    \alpha$ is $[F]$ itself. 

    We define $\cZ_3:=\bar \beta^{-1}(\cU)$, which is open in $\cZ_2$. On $\cZ_3\times X$
   we consider the dual sequence  \[
    0\to \cSS \to p_{\cZ}^*\cW\to \cH\to 0.
    \]
    By assumption, $p_{\cZ*}(\cSS)=0$. Hence we get an induced injective morphism of locally free sheaves on $\cZ_3$: \[
    \cW=p_{\cZ*}(p_{\cZ}^*\cW)\to p_{\cZ*}\cH,
    \]
    where $p_{\cZ*}\cH$ is locally free because it is a pullback from the universal sheaf on $\cU$.

\begin{definition}\label{defSyz}    We define the {\em  moduli stack}  $\SSyz^{w}_{\cU}$ {\em of generalized syzygy bundles}   to be the open substack in $\cZ_3$ where $\cW\to p_{\cZ*}\cH$ is an injective  morphism of vector bundles (i.e., where $p_{\cZ*}\cH^*\to \cW^*$ is surjective); it is a locally closed substack of $\SSpl(r';c_i')$.

 We define the {\em moduli space}  $\Syz^{w}_{U}$ {\em of generalized syzygy bundles} as the induced locally closed algebraic subspace of $\Spl(r';c_i')$.

 Moreover, for every  locally closed substack  $\cL\subset \cU$, we define $\SSyz^w_{\cL}:=\SSyz^w_{\cU}\cap  \bar\beta^{-1}(\cL)$. If we denote by $L$ the image of $\cL$ in $U$ (also locally closed), we define   $\Syz^w_{L}:=\Syz^w_{U}\cap \beta^{-1}(L)$. Note that $\Syz^w_L$ is the image in $\Spl(r';c_i')$ of $\SSyz^w_{\cL}$.
  \end{definition}

   By construction, the morphism $\bar \alpha$ (resp. $ \alpha $) factors via $\SSyz^{w}_{\cU}$ (resp. $\Syz^{w}_U$).
Note also that we have constructed a morphism $$\bar \beta:\SSyz^{w}_{\cU}\to \cU$$ defined by $\cH$ on $\SSyz^{w}_{\cU}\times X$ and  a rank $w$ subbundle $\cW$ of $p_{\cZ*}\cH$ on $\SSyz^{w}_{\cU}$. This lifts $\bar \beta $ to a unique morphism $$\bar \gamma:\SSyz^{w}_{\cU} \to \cG_\cU,$$ and in fact to $\cG^0_{\cU}$ because of the surjectivity of $p_{\cZ}^*\cW\to \cH$. 

\subsection{Proof of Theorem \ref{embedding}} 

To complete the proof, we need to show that $\bar \alpha:\cG^0_{\cU}\to \SSyz^{w}_{\cU}$ and $\bar \gamma:\SSyz^{w}_{\cU}\to \cG^0_{\cU}$ are  inverses to each other by finding canonical isomorphisms among the induced families of sheaves on $\cG^0_{\cU}\times X$ and $\SSyz^{w}_{\cU}\times X$, respectively. This will correspond to take the proof of Proposition \ref{inj} and write it again in families instead of for a single vector bundle.

Given  a scheme $B$; a morphism $f:B\to \cG^0_{\cU}$ is given by a sheaf $\cF\in \Coh(X\times B)$ defining a morphism $B\to \cU$ and a subbundle $\cW$ of rank $w$ in the bundle $p_{B*}{\cF}$, such that the induced morphism $p_B^*\cW\to \cF$ fits into an exact sequence on $X\times B$ \[
0\to \cS\to p_B^*\cW \to \cF\to 0.
\]

The morphism $\bar \alpha(f):B\to \SSyz^{w}$ is defined by the sheaf $\cS$ (which is thus by definition $(\id_X\times \bar \alpha(f))^*\tilde\cS$).

We define $\tilde \cW$ on $X\times B$ as the dual of $p_{B*}(\cS^*)$, and $\cH$ the locally free sheaf on $X\times B$ dual to the kernel of the induced morphism $p_B^*\tilde\cW^*\to \cS^*$. We also know that the induced morphism $\tilde\cW\to p_{B*}\cH$ is injective as a morphism of vector bundles.
The morphism $\bar \gamma(\bar \alpha(f)):B\to \cG^0_{\cU}$ is defined by the pair $(\cH,\tilde\cW\subset p_{B*}\cH)$. We need to construct explicit isomorphisms $\cW\to \tilde \cW$ and $\cF\to \cH$ compatible with the inclusions  $\cW\subset p_{B*}\cF$ and $\tilde\cW\subset p_{B*}\cH$, and which commute with arbitrary pullback $u:B_1\to B$.

The isomorphism $\tilde \cW\to  \cW$ is induced by the isomorphism $\cW^*\to p_{B*}\cS=:\tilde\cW^*$. This induces an exact sequence \[
0\to \cH^* \to p_{B}^*\cW^*\to \cSS^*\to 0.
\]
whose dual induces an isomorphism $\cF\to \cH$. All constructions are natural, thus commute with pullback and this concludes the proof.

\begin{corollary} The morphism $\bar \beta\circ \bar \alpha:\cG^0_{\cU}\to \cU$  is naturally isomorphic to the projection $\pi:\cG^0_{\cU}\to \cU$. Therefore, for every locally closed substack $\cL$ of $\cU$, the restriction of $\bar \alpha$ induces an isomorphism $\cG^0_{\cL}\to \Syz^{w}_{\cL}$, where $\cG^0_{\cL}$ is defined to be $\pi^{-1}(\cL)$, and thus also an isomorphism $\cG^0_L\to \Syz^w_L$.
\end{corollary}

\section{Proof of Theorem \ref{syz_open}}

This proof is independent of subsections 4.2 and 4.3. In view of Proposition \ref{inj}, it is enough to show that $\bar \alpha$ is \'etale and, by Proposition \ref{dif_inj}, it is enough to show that $\bar \alpha$ is smooth. 
We define $$\widetilde{\cG_U^0}:=\{(F,W,s_1,\cdots ,s_w) \mid (F,W)\in \cG_U^0, \, s_1,\cdots ,s_w \text{ basis of } W\},$$ the principal frame bundle associated to the vector bundle $\cW$ on $\cG^0_{\cU}$, and denote by $q:\widetilde{\cG_\cU^0}\to  \cG_\cU^0$ the natural projection. It is a principal $PGL(w)$-bundle. In particular it is smooth and surjective. Therefore, to prove that $\bar \alpha$ is smooth, it is enough to show that $\tilde\alpha:=\bar \alpha\circ q$ is smooth.

 The surjection $$(s_1,\ldots,s_w):\cO_{X\times \cG_\cU^0}^{\oplus w}\simeq q^*\cW\longrightarrow q^*\cF$$ on $X\times \widetilde{\cG_\cU^0}$defines a morphism $\widetilde{\cG_\cU^0}\to\Quot^P(\cO_X^w)$ and thus $\widetilde{\cG_U^0}\to \Quot^P(\cO_X^w)$, where $P$ is the Hilbert polynomial of any $[F]\in \Spl(r;c_i)$. Its image is contained in the open locus $Q_{spl}$  in $\Quot^P(\cO_X^w)$ parametrizing quotients that are simple vector bundles of rank $r$ and Chern classes $c_i$. The universal quotient bundle defines a morphism $Q_{spl}\to \Spl(r;c_i)$ and we denote by  $Q_U$ the inverse image of $U$ intersected with the open locus of points $(F,s_1,\ldots,s_w)$ where the  $s_i$ are linearly independent in $H^0(F)$. 

 By construction, the induced morphism $\cG_U^0\to Q_U$ is an isomorphism and $\cG^0_U\to \Quot^P(\cO_X^w)$ is an open embedding. Hence we can use for $\cG_U^0$ the tangent and obstruction spaces at any point $(F,s_1,\ldots,s_w)$ of $\Quot^P(\cO^w_X)$, which are $H^0(S^*\otimes F)$ and $H^1(S^*\otimes F)$. We are now ready to prove Theorem 1.2.

\begin{proof} 
We already saw that it is enough to show that $\tilde\alpha$ is smooth and by Proposition \ref{criterion_smoothness} it is enough to  prove that  $\tilde{\alpha }$ induces surjective maps on tangent spaces and injective maps on obstruction spaces at every point $(F,s_1,\ldots,s_w)$ of $\cG^0_\cU$. The morphism on tangent and obstruction spaces is induced by taking cohomology of the short exact sequence $$
0\to S\otimes S^*\to  (S^*)^{\oplus w}\to F\otimes  S^*\to 0
$$
which induces an exact sequence \[
H^0(F\otimes S^*)\to H^1(S\otimes S^*)\to H^1(S^*)^{\oplus w}\to H^1(F\otimes S^*)\to H^2(S\otimes S^*)
\]
Therefore, to show smoothness, it is enough to prove that $H^1(S^*)=0$.\\
Taking cohomology in  the exact sequence 
$$
0 \to F^*\to \cO_X^{\oplus w}\to S^* \to 0,
$$
and using the hypotheses $H^1(\cO_X)=H^2(F^*)=0$
we deduce that $H^1(S^*)=0$, which concludes the proof.
\end{proof}

\begin{corollary}
    Assume that $h^1(\cO_X)=0$. If the  algebraic space $V$ is smooth (e.g., this is true if $V=[L]$ with $L$ an ample line bundle and $\dim X\ge 3$) then $\Syz^{w}_V$ is a smooth open subspace of $\Spl(r';c_i')$.
\end{corollary}

\section{Examples and counterexamples}

The first goal of this section is to  show that our assumptions are verified in a large number of significant cases. Thus we have a recursive construction of open or locally closed subspaces of the moduli of simple sheaves and, in particular, of smooth components with a very explicit description.

Our second goal is to illustrate with explicitly examples that all our  assumptions are necessary.

\subsection{Explicit constructions based on line bundles and rigid bundles}

We  start showing that our approach nicely applies in the particular case of $F$  an ample line bundle  $L$ and $H^1(\cO_X)=0$. Then, Theorem \ref{embedding} holds when $H^1(L)=0$, and Theorem \ref{syz_open} holds under the additional assumption that $\dim X\ge 3$, and the moduli so obtained are smooth.
 Indeed, we have:

\begin{corollary}\label{coro1} Let $X$ be a  projective variety of dimension $n\ge 2$ such that  $H^1\cO_X  =
 0$ and $L$ an ample line bundle with $H^1\cO_X(L)=0$. Let $W\subset H^0\cO_X(L)$  be a base point free linear system. Then the syzygy bundle $S$ associated to $(\cO_X(L),W) $ is unobstructed. Moreover, if $\dim X\ge 3$ then $S$ is rigid if and only if the linear system $W$ is complete (i.e. $W=H^0(\cO_X(L))$). 
\end{corollary}
\begin{proof} It immediately follows from Theorems  \ref{embedding} and \ref{syz_open}. This result was proven for $\dim X\ge 4$ in \cite{MS2}.
\end{proof}

More generally, our recursive structure can be carried over starting with  a  rigid (and in particular an exceptional) vector bundle on $X$. Indeed, we have:

\begin{corollary}\label{coro2} Let $X$ be a projective variety of dimension $n\ge 2$ such that  $H^1\cO_X  =
0$ and $F$ be a rigid bundle on $X$  with   $H^1(F)=H^1(F^*)=H^2(F^*)=0$ and generated by global sections. Let $W\subset H^0(F)$ a subspace  of dimension $w\ge \dim(X)+rk(F)$. Then the syzygy bundle $S$ associated to $(F,W) $ is unobstructed. Moreover, if $\dim X\ge 3$ then $S$ is rigid if and only if  $W=H^0(F)$.
\end{corollary}
\begin{proof} Again it is a direct application of Theorems  \ref{embedding} and \ref{syz_open}. 
\end{proof}

\begin{remark} \label{remcoro}
There is a wide range of projective varieties X satisfying that $H^1 \cO_X = 0$: projective spaces, hypersurfaces in $\PP^n$ with  $n\ge 4$, Grassmannians, flag varieties, complete smooth projective toric varieties, etc.
Therefore, the Corollaries \ref{coro1} and \ref{coro2} can be applied, for instance,  to the following cases:
\begin{enumerate}
    \item $X=\PP^n$, $n\ge 2$ and $F=\Omega^i_{\PP^n}(s)$ with $s\ge i+1$;
    \item $X=Q\subset \PP^n$, $n\ge 4$, a quadric hypersurface, $F=Spin\otimes \cO _Q(t)$ where $Spin $ is a Spinor bundle and $t\ge 0$ an integer.
    \item $X=Gr(k,n)$ and $F=\Sigma ^{\beta} \cQ\otimes \Sigma ^{\gamma}\cR\otimes \cO_X(t)$ where $\cQ$ and $\cR$ are the universal bundles on $X$ of rank $k+1$ and $n-k$, respectively; $\beta =(b_1,\dots ,b_{k+1})\in \ZZ^{k+1}$ and $\gamma =(a_1,\dots ,a_{n-k})\in \ZZ^{n-k}$ are two non-increasing sequences, $\Sigma ^{\beta}$ and $\Sigma ^{\gamma}$ are the Schur functors and $t\gg 0$.
    \item $X$ any smooth projective variety of dimension $n\ge 2$ satisfying that $H^1 \cO_X =
H^2\cO_ X  = 0$ and $F$ any exceptional bundle (i.e. $Hom(F,F)=\KK \id_X$ and $Ext^{i}(F,F)=0$ for $i>0$) on $X$ globally generated. 
\end{enumerate} 
\end{remark}

\subsection{Recursive construction of smooth rational moduli spaces}

In this subsection we assume that $\dim X\ge 3$ and $H^1(X,\cO_X)=0$. We will construct a recursive procedure to generate open subspaces in $\Spl(r;c_i)$ which are smooth, rational, quasiprojective schemes and have a universal bundle. We will then show that we can start the recursion based on the examples in the previous subsection. We keep the open subspace $V\subset \Spl(r;c_i)$ in the usual meaning.

\begin{proposition}
    Let $A\subset V$ be an open algebraic subspace that carries a universal family. Then \begin{enumerate}
        \item $\cG^0_A$ carries a universal family of syzygy bundles;
        \item if $A$ is smooth, then so is $\cG^0_A$;
        \item if $A$ is a quasiprojective scheme, then so is $\cG^0_A$;
        \item if $A$ is a quasiprojective, irreducible, smooth rational scheme, then so is $\cG^0_A$.
    \end{enumerate}
\end{proposition}
\begin{proof}
(1) A universal family on $A$ is equivalent to  a section $s:A\to \cU$ such that $\eps\circ s:A\to U$ is the inclusion. In particular, we can pullback via this section $\cG_{\cU}$ to $s^*\cG_{\cU}$. by the universal property of fibered products we get an isomorphism $s^*\cG_{\cU}\to \cG_A$, which induces an isomorphism $s^*\cG^0_{\cU}\to \cG^0_A$.
        
         (2) This follows from the smoothness of $\pi:\cG^0_U\to U$.
        
         (3) By construction $\cG^0_A$ is open in $s^*\cG_{\cU}$, and $s^*\cG_{\cU}$ is locally a Grassmann bundle over $\cU$; since $A$ is quasiprojective, it follows that $\cG_A\to A$ is globally a Grassmann bundle and thus quasiprojective.
        
        (4) The Grassmann bundle $\cG_A$ over $A$ is  quasiprojective, irreducible, smooth rational scheme because so is $A$; the result follows because $\cG^0_A$ is open in $\cG_A$.
\end{proof}

\begin{remark}
    The proposition gives us a recursive procedure because $\Syz^w_A$, the image of $\cG^0_A$ in $\Spl(r';c'_i)$, is a moduli space of simple bundles and after twisting we may assume that all its points satisfy the required vanishing. Note that the property of having a universal family is preserved by twisting.
\end{remark}

\begin{remark}
    To start the recursive construction we can use an ample line bundle $L$ which is generated by global sections and such that $H^1(L)=0$ (the other vanishings follow from Kodaira vanishing). We can also use any rigid simple bundle $F$ which is generated by global sections and such that $H^1(F)=H^1(F^*)=H^2(F^*)=0$; this can be obtained by sufficiently twisting any rigid simple vector bundle (eg, a syzygy bundle as shown in Corollary \ref{coro1} and Corollary \ref{coro2}).
\end{remark}

The case where $X$ is a Calabi-Yau 3-fold is particularly nice since the vanishing of $H^1(F)$ is equivalent to the vanishing of $H^2(F^*)$; in particular, all vanishings hold for an ample line bundle. 

\begin{corollary} Let $X$ be a projective Calabi-Yau 3-fold. Assume that $H^1(\cO_X)=H^2(\cO_X)=0$ (for instance, take $X$ a quintic 3-fold in $\PP^4$). Let $L$ be a very ample line bundle on $X$ such that the generic $w$-dimensional $W\subset H^0(\cO_X(L))$ is a base point free linear system. Then $\Syz^{w}\subset \Spl(r';c_i')$ is an open subspace, which is smooth, rational, quasiprojective of dimension $w(h^0(\cO_X(L))-w)$.
    \end{corollary}
It is worthwhile to point out that in this last case the moduli space is smooth in spite of the fact that $Ext^2(S,S)=Ext^1(S,S)\ne 0$ for any $S\in \Syz^{w}$. 

\subsection{What happens if we drop our assumptions?}
We now show that our assumptions are necessary, by explicitly discussing some examples where the syzygy bundle fails to be simple (because the hypothesis $H^1(F^*)=0$ is not fulfilled)  or some examples where $\alpha$ is not an open embedding when $\dim X=2$ or $H^2(F^*)\ne 0$. Let first observe that our assumptions are rather mild.

\begin{remark}\label{assumptions}
     For every line bundle $L$ and integer $N$,  tensoring with $\cO_X(NL)$ induces an isomorphism $$\Tw_N:Spl(r;c_{i})\cong Spl(r; \tilde { c_{i}})$$ which maps $F$ to $F(NL):=F\otimes \cO_X(NL)$; here $ \tilde { c_{i}}:=c_i(F(NL))$.

     For any sheaf $F$, if we choose $L$ ample and $N$ sufficiently big,  $F(NL)$ is generated by global sections and $H^1(F(NL))=H^1((F(NL))^*)=0$. So the assumption that $\cU$ be nonempty is not very restrictive, since the theorem A and B of Serre guarantees that for any ample line bundle $L$ on $X$, there always exists an integer $N_0$ such that for all $N\ge N_0$,   $F (NL)$ is generated by global sections
    and $H^1(F(NL))=H^1(F(NL)^*)=0.$ 
    
     In general, the moduli space $\Spl(r;c_i)$ is not of finite type (i.e., quasi-compact), and it usually contains sheaves which are not vector bundles. This argument shows that there are no other limitations to applying our construction if we twist enough. That is, for every $U\subset \Spl(r;c_i)$ open quasi-compact subspace of $\Spl(r;c_i)$ containing only vector bundles, there is an $N_0$ such that twisting for any $N\ge N_0$ and letting $v:=\chi(F(N))$ for any $F\in \Spl(r;c_i)$, we have that $\tilde U\subset \Spl(r;\tilde{c_i})$ contains $\Tw(U)$.
\end{remark}

Let us now give an example which shows that the hypothesis  $H^1(F^*)=0$ is necessary to guarantee the simpleness of the generalized syzygy bundle.

\begin{example} \label{exp2} Fix $X=\PP^2$ and $M:=M(2;4,12)$ the moduli space of rank 2 stable vector bundles $F$ on $\PP^2$ with Chern classes $(c_1,c_2)=(4,12)$. A general vector bundle $F\in M$  has a resolution \cite[Theorem 4.1]{DM}:
\begin{equation}
\label{defF}
0 \to \cO_{\PP^2} (-2)^2\to \cO_{\PP^2}^4 \to F\to 0. 
\end{equation}
The bundle $F$ is generated by global sections, $H^0(F)=\KK^4$ and, since stability implies simpleness, we have that $F$ is simple.  The generalized syzygy bundle  $S$ associated to $(F,W=H^0(F))$ splits: $S=\cO_{\PP^2}(-2)^2$ and therefore $S$ is not simple ($H^0(S\otimes S^*)=\KK^4$). Let us check that $H^1(F^*)\ne 0$. If we dualize the exact sequence (\ref{defF}) we get:
$$
0\to F^* \to \cO_{\PP^2}^4 \to \cO(2)_{\PP^2}^2 \to 0.
$$
Taking cohomology we get
$$
0\to H^0(F^*) \to H^0(\cO _{\PP^2})^4\cong \KK^4 \to H^0(\cO _{\PP^2}(2))^2 \cong \KK^{12} \to H^1(F^*)
$$
and we conclude that $H^1(F^*)\ne 0.$
\end{example}

The assumption $H^1(\cO_X)=0$ is necessary for the validity of Theorem \ref{syz_open}, as shown by the following result.
We remark that for any choice of $c_1$, the moduli $\Spl(1;c_1,0,\ldots,0)$ is canonically isomorphic to $\Pic^{c_1}(X)$.

\begin{lemma}\label{dim3_1}
    Let $\dim X\ge 3$, $r=1$, and $c_1$ a very ample class such that $c_1+c_1(X)$ is also ample. Assume that $H^1(X,\cO_X)\ne 0$. Then for $F\in \Pic^{c_1}(X)$ and $W\subset H^0(F)$ generating $F$, the differential map  $d\alpha :   T_{[(F , W)]}\cG_U^0  \to T_{[S]}(\Spl(r';c_{i}')) 
$ is not surjective. 
\end{lemma}
\begin{proof} Notice that since $c_1$ a very ample class and  $c_1+c_1(X)$ is  ample, we have  $H^i(F)=0$ for every $i>0$ and every $F$ in $\Pic^{c_1}(X)$. Let us  consider the exact sequence
\begin{equation}\label{defS}
0\to S\to W\otimes \cO_X\to F\to 0.   
\end{equation}
To prove that the differential map $d\alpha$ is not surjective it is enough to check that $\dim H^1(S\otimes S^*)>\dim H^0(S^*\otimes F)/W\otimes W^*.$ To this end, we dualize the exact sequence (\ref{defS}), we tensor with $F$, we take cohomology and we get
$$
0\to H^0(\cO_X)\to W^*\otimes F \to H^0(S^*\otimes F)\to H^1(\cO_X)\to 0.
$$
Therefore, we have $h^0(S^*\otimes F)= wh^0(F)+h^1(\cO_X)-1$ and, hence, we obtain $\dim H^0(S^*\otimes F)/W\otimes W^*=
w(h^0(F)-w)+h^1(\cO_X)-1.$ 

    Let us know compute $\dim H^1(S\otimes S^*)$. First of all, using the exact cohomology sequence associated to (\ref{defS}) we get 
    $$h^1(S)=h^0(F)-w+wh^1(\cO_X).$$ From the exact sequence
    $$
    0\to S\otimes F^*\to W^*\otimes S\to S\otimes S^*\to 0
    $$
    we get the exact sequence
    $$
    0\to W^*\otimes H^1(S)\to H^!(S\otimes S^*)  \to H^2(S\otimes F^*)\to W^*\otimes H^2(S) 
    $$
which gives us $$\begin{array}{rcl} \dim H^1(S\otimes S^*)& \ge &  wh^1(S) \\ 
& = & w(h^0(F)-w)+w^2h^1(\cO_X)) \\ & > & \dim H^0(S^*\otimes F)/W\otimes W^*
\end{array}$$ which proves what we want    
\end{proof}
\begin{lemma} \label{dim3_2} Let $\dim X> 3$, $r=1$, and $c_1$ a very ample class such that $c_1+c_1(X)$ is also ample. Assume that $H^1(X,\cO_X)\ne 0$ and $H^2(\cO_X)=0$. Then for $F\in \Pic^{c_1}(X)$ and $W\subset H^0(F)$ generating $F$,
    the image of $\alpha$ is contained in the smooth locus of  moduli $\Spl(r'; c_i')$. 
\end{lemma}
\begin{proof}
    By Proposition \ref{tgt-obs} we only need to check that $Ext^2(S,S)=0$ and, using the exact sequence
    $$
    0\to S\otimes S^*\to W\otimes S^*\to F\otimes S^*\to 0
    $$
    it suffices to see that $H^1(S^*\otimes F)=H^2(S^*)=0$. From the exact sequence
    $$
    0\to F^*\to W^*\otimes \cO_X \to S^*\to 0
    $$
    we get $H^2(S^*)=0  $ and from 
    $$
    0\to F^*\otimes F\to W^*\otimes F \to S^*\otimes F\to 0
    $$
    we deduce that $H^1(S^*\otimes F)=0$ which concludes the proof.
\end{proof}

In next example, we illustrate by means of a concrete example that without the hypothesis $H^2(F^*)=0$ the morphism $\alpha $ is no longer an open  embedding. We show how in the particular case $X=\PP^2$ and $F$ a line bundle all construction can be explicitly described and the normal bundle to the syzygy locus can be computed.

\begin{example} We fix $X=\PP^2$, $F=\cO _{\PP^2}(3)$ and $w=3$; and we will construct $\Syz^3$ inside $\Spl(2;-3,9)\cong M(2;-3,9)$ (in $\PP^2$, the moduli of rank 2 simple bundles and the moduli space of rank 2 stable bundles coincide).
Consider $U\subset Gr(3,H^0(\cO_{\PP^2}(3)))$ the open subset such that for any $W\in U$ the evaluation map $W\otimes \cO_{\PP^2}\to \cO_{\PP^2}(3)$ is surjective, i.e. $U=\{\ (f_1,f_2,f_3)\in H^0(\cO_{\PP^2}(3)) \text{ without common zeros } \}$. For any $W\in U$  we have the associated syzygy bundle $S$ which fits in the exact sequence
\begin{equation}\label{auxaux}
0 \to S \to W\otimes \cO_{\PP^2} \to \cO_{\PP^2}(3) \to 0.
\end{equation}
Moreover we know:
\begin{itemize}
\item[(i)] $S\in Spl(2;-3,9)$.
\item[(ii)] $Ext^2(S,S)=H^2(S\otimes S^*)=0$ and, hence $S$ is unobstructed,
\item[(iii)] $\dim Syz^3(2;-3,9)=\dim Gr(3, H^0(\cO_{\PP^2}(3)))=21$.
\end{itemize}
\noindent {\bf Claim:} The fiber of the conormal bundle of $Syz^3(2,-3,9)$ in $\Spl(2;-3,9)$ at the point $[S]$ is isomorphic to $H^2(S(-3))$. So, the codimension of $Syz(2;-3,9)$ in $\Spl(2;-3,9)$ is $h^2(S(-3))=3$.

\vskip 2mm
\noindent {\bf Proof of the Claim:} Since $Ext^2(S,S)=0$ and $Hom(S,S)=\KK$ we have $$\begin{array}{rcl} \dim_{[S]}Spl(2;-3,9) & =  & \dim Ext^1(S,S) \\
& = & 1-\chi (S\otimes S^*)\\
& = & 24.
\end{array}
$$
On the other hand, we dualize the exact sequence (\ref{auxaux}), we twist it with $S$, we take cohomology and we obtain:
$$
0\to H^0(S\otimes S^*)\to H^1(S(-3))\to W^*\otimes H^1(S)\to H^1(S\otimes S^*)\to H^2(S(-3))\to 0.
$$
This last exact sequence splits into
$$
0\to \KK\cong H^0(S\otimes S^*)\to H^1(S(-3))\cong \KK \to 0, \text{ and }$$
$$
0\to W^*\otimes H^1(S)\to H^1(S\otimes S^*)\to H^2(S(-3))\to 0 $$
which proves our claim.

Let us describe using Serre's correspondence first a general point $S$ in the moduli space $\Syz^3$ and then a general point $E$ in $\Spl(2;-3,9)$. If $S$ is a point in $\Syz^3$, then $h^0(S^*)=3$ and $H^0(S^*(-1))=0$. Therefore, a general section $s\in H^0S^*$ vanishes in codimension 2 and gives rise to an exact sequence
$$
0\to \cO_{\PP^2}\to S^*\to {\mathcal I}_Z(3)\to 0
$$
where $Z\subset \PP^2$ is a 0-dimensional subscheme of length 9 complete intersection of 2 plane cubics. Indeed, $h^0(S^*)=3$ implies that $h^0({\mathcal I}_Z(3))=2$. To construct a general bundle $E\in \Spl(2;-3,9)$ we take a general set $Y\subset \PP^2$ of 9 different points (in particular, not a complete intersection) and   a non-zero extension $0\ne e\in Ext^1({\mathcal I}_Y(3),\cO_{\PP^2})$. Via Serre's construction we build a rank 2 vector bundle $E^*$ on $\PP^2$ which sits in the exact seqeunce:
$$
0\to \cO_{\PP^2}\to E^*\to {\mathcal I}_Y(3)\to 0.
$$
Therefore, $E$ is a rank 2 stable (hence, simple) vector bundle with Chern classes $(-3,9)$, i.e. $E\in \Spl(2;-3,9)$. By assumption $h^0(\mathcal{I}_Y(3))=1$, hence $h^0(E^*)=2$.

From the exact sequence \[ 0\to \mathcal{I}_Y\to \cO_{\PP^2}\to \cO_Y\to 0
\]
it follows that $\dim Ext^1({\mathcal I}_Y(3),\cO_{\PP^2})=\dim H^1(\mathcal{I}_Y)=8$; hence
 the family of bundles we have just constructed has dimension $\dim \Hilb^9(\PP^2)-\dim H^0(E^*)+\dim Ext^1({\mathcal I}_Y(3),\cO_{\PP^2})= 24.$ Therefore, $E$ lies in an open dense subset of $Spl(2;-3,9)$.

Notice that the main difference between a general syzygy bundle $S\in Syz^3(2;-3,9)$ and a rank 2 simple bundle $E\in \Spl(2;-3,9)$ lies in the fact that in the first case the associated 0-dimensional scheme is a complete intersection of 2 plane cubics while in the second case is a general set of 9 points in $\PP^2$.
\end{example}

\subsection{Very generalized syzygy bundles}\label{very_gen}

We now want to show that if we generalize further our definition of syzygy bundles, by allowing kernels of surjective morphisms $E\to F$, where $E$ is a rank $w$ vector bundle which is a deformation of $\cO^{\oplus w}_X$, then we get an \'etale morphism to the moduli (although possibly not an injective one).

Let $\VB_X(w;0)$ be the algebraic stack parametrizing all vector bundles on $X$ of rank $w$ and Chern classes zero; it contains $\cO_X^w$. We consider now the moduli stack $\GSyz:=\GSyz^w(r;c_i)$ parametrizing all surjective morphisms $\phi:E\to F$ where $E\in \VB_X(w;0)$, $F\in \cU$,  $H^i(E\otimes F^*)=0$ for $i=0,1,2$ and $S:=\ker \phi$ is a simple vector bundle. Note that this is open in the relative quotient stack of the universal bundle on $\VB_X(w;0)$.

Over $X\times \GSyz$ we have a universal sequence of vector bundles \[
0\to \cS \to \cE\to \cF\to 0
\]
where $\cE$ (resp~$\cF$) is the pullback of the universal bundles from $X\times VB_X(r;0)$ (resp.~$X\times \cU$). The bundle $\cS$ defines a morphism $$\lambda:\GSyz\to \Spl(r';c_i').$$ We want to show that $\lambda$ is an \'etale morphism of algebraic stacks. We define a complex $A\in C^{0,1}(\Coh(X))$ as follows: \[
A:=[E^*\otimes E\oplus F^*\otimes F\to E^*\otimes F ]
\]
where the morphism is $\phi \circ \cdot$ on $E^*\otimes E$ and $-(\cdot\circ \phi)$ on $F^*\otimes F$.

We also define a complex $B\in C^{-1,1}(\Coh(X))$ as follows: \[
B:=[F^*\otimes E\to E^*\otimes E\oplus F^*\otimes F\to E^*\otimes F ]
\]
where the first morphism is $(-\cdot\circ \phi,\phi \circ \cdot)$ and the second is the same as in $A$.
There is an exact sequence of complexes in $C^b(\Coh(X))$:
\begin{equation}\label{ex-compl}
    0\to A\to B \to (F^*\otimes E)[1]\to 0.
\end{equation}

\begin{lemma}
    For every $i\ge 0$, there is a natural isomorphism $\HH^i(B)\to H^i(S^*\otimes S)$.
\end{lemma}
\begin{proof}
 We strongly use that if $E^\bullet$ is a locally free resolution of a coherent sheaf $F$, then $\HH^i((E^\bullet)^*\otimes E^\bullet)$ is naturally isomorphic to $\Ext^i(F,F)$. The complex $B':=[F^*\to E^*]\in C^{-1,0}(\Coh(X))$ is a locally free resolution of $S^*$, and $B$ is the complex associated to the double complex $B_0\otimes B_0^*$.
Hence $\HH^i(B)=\Ext^i(S^*,S^*)=H^i(S^*\otimes S)$.
\end{proof}

Let $\phi:E\to F$ on $X$ be a point $p$ in $\GSyz$. By Corollary \ref{def-triple}, for $i=0,1,2$ we have $T^i_p\GSyz=\HH^i(X,A)$.

\begin{proposition}
    The morphism $\lambda:\GSyz\to \Spl(r';c_i')$ is representable and \'etale.
\end{proposition}
\begin{proof}
    Let $\phi:E\to F$ be a point in $\GSyz$. Then $T^i(\lambda)$ is induced by the natural morphism $A\to B$ above, where $T^i_S=H^i(S^*\otimes S)=\HH^i(B)$.\\
    The exact sequence of complexes (\ref{ex-compl}) induces a long exact sequence in cohomology\[
    H^i(F^*\otimes E)\to \HH^i(A)\to \HH^i(B)\to H^{i+1}(F^*\otimes E).
    \]
    The assumption that $H^i(F^*\otimes E)=0$ for $i=0,1,2$ implies that $T^i\lambda$ is an isomorphism for $i=0,1$ and injective for $i=2$. Therefore, $\lambda $ is smooth, DM type, and \'etale. Since the endomorphism group of $\phi:E\to F$ is a vector space, it is naturally isomorphic to its tangent space at the identity which is one-dimensional by $T^0\lambda$ being an isomorphism. Thus  $\Aut(\phi:E\to F)=\KK(\id_E,\id_F)$ and the morphism $\lambda$ sends it isomorphically to $\KK\id_S=\Aut(S)$.
\end{proof}

\end{document}